\documentclass{dmtcs}
\usepackage{amsmath,mathtools,amssymb,amscd}
\usepackage[backend=bibtex]{biblatex}
\usepackage{xspace}
\usepackage{tikz}
\usetikzlibrary{shapes,fit}
\tikzstyle{block}=[ellipse, fill=blue, opacity=0.3]
\tikzstyle{Xsize}=[minimum size=1.5cm]

\let\x\proof\let\y\endproof
\let\proof\relax\let\endproof\relax
\usepackage{amsthm}
\let\proof\x\let\endproof\y
\newtheorem{thm}{Theorem}[section]
\newtheorem*{thm*}{Theorem}

\newtheorem{lemma}[thm]{Lemma}
\newtheorem{prop}[thm]{Proposition}

\theoremstyle{definition}
\newtheorem{defn}[thm]{Definition}
\newtheorem{ex}[thm]{Example}
\newtheorem{rem}[thm]{Remark}

\newcommand{\Set}[1]{\ensuremath{\mathcal{#1}}} 
\newcommand{\Dfn}[1]{\emph{#1}} 

\newcommand{\eval}[2][\right]{\relax
  \ifx#1\right\relax \left.\fi#2#1\rvert}

\newcommand{\Sp}{\mathrm{Sp}}
\newcommand{\GL}{\mathrm{GL}}

\newcommand{\Hom}{\mathrm{Hom}}
\newcommand{\End}{\mathrm{End}}
\newcommand{\bN}{\mathbb{N}}
\newcommand{\bC}{\mathbb{C}}
\newcommand{\bZ}{\mathbb{Z}}
\newcommand{\Br}{\mathsf{Br}} 
\newcommand{\cJ}{\mathsf{J}} 
\newcommand{\T}{\mathsf{T}} 
\newcommand{\fS}{{\mathfrak S}}
\newcommand{\pr}{\mathrm{pr}}

\newcommand{\dc}{\mathsf{D}} 

\newcommand{\Pf}{\mathrm{Pf}}

\DeclareMathOperator{\ev}{ev}
\DeclareMathOperator{\bev}{\overline{ev}} 

\DeclareMathOperator{\maj}{maj}

\DeclareMathOperator{\id}{id}

\DeclareMathOperator{\tr}{\mathbf{tr}}   
\DeclareMathOperator{\ch}{\mathbf{ch}}   
\DeclareMathOperator{\fd}{\mathbf{fd}}   

\title{Combinatorics of symplectic invariant tensors}%

\author{Martin Rubey\addressmark{1}\thanks{Email:
    \email{Martin.Rubey@tuwien.ac.at}}\and%
  Bruce W. Westbury\addressmark{2}\thanks{Email:
    \email{Bruce.Westbury@warwick.ac.uk}}}
\address{\addressmark{1}Fakult\"at f\"ur Mathematik und
  Geoinformation, TU Wien, Austria\\
  \addressmark{2}Department of Mathematics, University of Warwick,
  Coventry, CV4 7AL}%

\begin{document}
\maketitle
\begin{abstract}
\paragraph{Abstract.}
An important problem from invariant theory is to describe the
subspace of a tensor power of a representation invariant under the
action of the group.  According to Weyl's classic, the first main
(later: \lq fundamental\rq) theorem of invariant theory states that
all invariants are expressible in terms of a finite number among
them, whereas a second main theorem determines the relations between
those basic invariants.

Here we present a transparent, combinatorial proof of a second
fundamental theorem for the defining representation of the symplectic
group $\Sp(2n)$.  Our formulation is completely explicit and provides
a very precise link to $(n+1)$-noncrossing perfect matchings, going
beyond a dimension count.  As a corollary, we obtain an instance of
the cyclic sieving phenomenon.

\paragraph{R\'esum\'e.}

Une probl\'eme importante de la th\'eorie des invariantes est de
d\'ecrire le sous espace d'une puissance tensorielle d'une
r\'epresentation invariant \`a l'action de la groupe.  Suivant la
classique de Weyl, la th\'eoreme fondamentale premiere pour la
r\'epresentation standard de la group sympl\'ectique dit que toutes
invariantes peuvent \^etre expriment entre un nombre fini d'entre
eux.  Ainsi, une th\'eoreme fondamentale seconde determine les
r\'elations entre ces invariantes basiques.

Ici, nous pr\'esentons une preuve transparente d'une th\'eoreme
fondamentale seconde pour la r\'epresentation standard de la groupe
sympl\'ectique $\Sp(2n)$.  Notre formulation est completement
explicite est elle provide un lien tres pr\'ecis avec les couplages
parfaites $(n+1)$-noncroissants, plus pr\'ecis qu'un denombrement de
la dimension.  Comme corollaire nous exhibons une ph\'enom\`ene du
crible cyclique.
\end{abstract}

\begin{center} Dedicated to Mia and George
\end{center}

\section{Introduction}
The primary motivation of this article is a specific example of the
cyclic sieving phenomenon, as introduced by Reiner, Stanton and
White~\cite{MR2087303}, concerning the set $X(r,n)$ of
$(n+1)$-noncrossing perfect matchings of $\{1,\dots,2r\}$.  These are
perfect matchings that do not contain any set of $(n+1)$ pairs
$(a_1,b_1),\dots,(a_{n+1},b_{n+1})$ with $a_1<\dots<a_{n+1}<b_1<\dots
<b_{n+1}$.

A standard visualisation of a perfect matching is obtained by placing
the numbers from $1$ to $2r$ in this order on a circle and connecting
two numbers by a straight edge if they form a pair in the matching.
Using this visualisation, a perfect matching is $(n+1)$-noncrossing
if at most $n$ edges cross mutually.  The set $X(r,n)$ carries a
natural action of the cyclic group of order $2r$, given by $\rho:
i\mapsto i\pmod{2r}+1$, i.e., the rotation map.

With this definition we can state what might be taken as the main
theorem of this article:
\begin{thm*} Let
  \begin{equation*}
    P(q) = \sum_T q^{\maj T},
  \end{equation*}
  where the sum is over all $n$-symplectic oscillating tableaux of
  length $2r$ and weight $0$ and $\maj T$ is the sum of the positions
  of the descents in $T$, as defined in~\cite{MR3226822}.

  Then the triple $\big(X(r,n), \rho, P(q)\big)$ exhibits
  the cyclic sieving phenomenon.
\end{thm*}

However, we would like to stress that this is merely a by-product of
the approach taken in this article to the invariant theory of the
classical groups, aiming at tying together their representation
theory and combinatorics in a concrete fashion.

In our main example we consider tensor powers of the defining
representation $V$ of the symplectic group $\Sp(2n)$, the group
acting diagonally.  The symmetric group $\fS_r$ also acts on $\otimes^r V$,
by permuting tensor positions.  This action is inherited by the
subspace of $\otimes^r V$ which is invariant under the action of
$\Sp(2n)$.

An effective method to understand both the tensor powers $\otimes^r
V$ and their invariant subspaces is to use diagram categories.  In
the case at hand this is a specialisation $\dc_{\Sp(2n)}$ of the
Brauer category.  The objects of this category are the natural
numbers and $\Hom_{\dc_{\Sp(2n)}}(r, s)$ is the vector space whose
basis is the set of perfect matchings of $\{1,\dots,r+s\}$, with a
natural composition, see
Section~\ref{sec:fundamental-theorems-Brauer}.

The connection to the tensor powers of the defining representation of
$\Sp(2n)$ is established by considering the category $\T_{\Sp(2n)}$
of invariant tensors of the defining representation of the symplectic
group, and a certain functor $\ev_{\Sp(2n)}$ from the diagram
category $\dc_{\Sp(2n)}$ to $\T_{\Sp(2n)}$, see
Definition~\ref{defn:ev}.

The \emph{first fundamental theorem}, Weyl~\cite[Theorem
(6.1A)]{MR1488158}, due to Brauer \cite{MR1503378},
is equivalent to the statement that this functor
is full, that is, surjective on morphisms.  What is traditionally
named \emph{second fundamental theorem} is a description of the
kernel of the linear map from $\Hom_{\dc_{\Sp(2n)}}(r, s)$ to
$\Hom_{\T_{\Sp(2n)}}(r, s)$ induced by $\ev_{\Sp(2n)}$.  We provide
the following explicit formulation, which is also a key ingredient
for the proof of Theorem~\ref{thm:pmcsp}.
\begin{thm*}
  \renewcommand*{\thefootnote}{$\dagger$} Let
  $\Pf\in\Hom_{\dc_{\Sp(2n)}}\big(0,2(n+1)\big)$ be the sum of all
  perfect matchings of $\{1,\dots,2(n+1)\}$ and let $\Pf^{(n)}$ be
  the \lq pivotal symmetric\rq\footnote{%
    We define pivotal and symmetric ideals in~\cite{RubeyWestbury},
    and restrict ourselves here to down-to-earth language in
    Definitions~\ref{dfn:Pf-diagrammatic} and~\ref{dfn:Pf-ideal}.}\
  ideal generated by $\Pf$.  Then the categories
  $\dc_{\Sp(2n)}/\Pf^{(n)}$ and $\T_{\Sp(2n)}$ are isomorphic.
  Moreover, the set $X(r,n)$ is a basis of
  $\Hom_{\dc_{\Sp(2n)}/\Pf^{(n)}}(0,2r)$.
\end{thm*}
Remarkably, a well-known result due to Sundaram~\cite{MR2941115}
states that the cardinality of $X(r,n)$ equals the dimension of
$\Hom_{\Sp(2n)}(\bC, \otimes^{2r} V)$.  In fact, Sundaram also
computed the Frobenius character of this representation, and, more
generally, of $\Hom_{\Sp(2n)}\big(W(\mu), \otimes^{r} V\big)$ for any
irreducible representation $W(\mu)$ of the symplectic group
$\Sp(2n)$.

We note that the same categorical setup was employed by Lehrer and
Zhang~\cite{LehrerZhang} to show that $\Pf^{(n)}$ restricted to the
Brauer algebra $\Hom_{\dc_{\Sp(2n)}}(r, r)$ is in fact generated as a
two-sided ideal by a single element.

\section{Fundamental theorems}
\label{sec:fundamental-theorems-Brauer}
The theorems of this section connect the Brauer category with the invariant
theory of the defining representation of the symplectic group by providing
an explicit isomorphism of categories.

First we recall the definition of the Brauer category.  Let $D(r,s)$
be the set of perfect matchings on $[r]\amalg [s]$.  In particular,
$D(r,s)=\emptyset$ if $r+s$ is odd.  An element of $D(r,s)$ is
visualised as a set of $(r+s)/2$ strands drawn in a rectangle, with
$r$ endpoints on the top edge of the rectangle and $s$ endpoints on
the bottom edge.

The composition of two diagrams $x\in D(r,s)$ and $y\in D(s,t)$
is obtained by identifying the points on the
bottom edge of $x$ with the points on the top edge of $y$.
Let $c(x,y)$ be the number of closed loops in the resulting diagram
and let $x\circ y \in D(r,t)$ be the perfect matching obtained
by removing all closed loops.

The Brauer category $\dc_{\Br}$ has objects $\bN$.  Its morphisms
$\Hom_{\dc_{\Br}}(r,s)$ are given by the free $\bZ[\delta]$-module
with basis $D(r,s)$.  The composition of morphisms is defined on
basis elements by
\begin{align}\label{eq:multiply}
  x\cdot y = \delta^{c(x,y)}x\circ y
\end{align}
and extended bilinearly.  The endomorphism algebra
$\Hom_{\dc_{\Br}}(r,r)$ is the Brauer algebra, which will be denoted
by $D_r$ from now on.  Finally, let $\dc_{\Sp(2n)}$ be the
specialisation of $\dc_{\Br}$ obtained by applying the homomorphism
$\bZ[\delta]\rightarrow\bC$, $\delta\mapsto -2n$.

The Brauer category also has a tensor product.  The tensor product
$x\otimes y$ of two diagrams $x$ and $y$ is obtained by putting the
diagrams side-by-side, identifying the right hand edge of $x$ with
the left hand edge of $y$.

We now turn to the definition of the category of invariant tensors
$\T_{\Sp(2n)}$.  Recall that the symplectic group $\Sp(V)$ is the
group of linear transformations of a $2n$-dimensional complex vector
space $V$ that preserve a non-degenerate skew-symmetric bilinear form
$\langle\;,\;\rangle$.  However, the categorical machinery we use
requires that a \emph{symmetric} bilinear form is preserved.
Therefore, we regard $V$ as an \emph{odd} vector space, which has
precisely the desired effect.  Among the further consequences we
remark that in this convention the dimension of $V$ is $-2n$ and
symmetric and exterior powers are interchanged.

Let $\{b_1,\dotsc,b_{2n}\}$ be a basis of $V$ and let $\{\bar
b_1,\dotsc,\bar b_{2n}\}$ be the dual basis, so $\langle \bar
b_i,b_j\rangle = \delta_{i,j}$.  Let $\T_{\Sp(2n)}$ be the category
whose objects are again the natural numbers and whose morphisms
$\Hom_{\T_{\Sp(2n)}}(r,s)$ are the equivariant maps
\[
\Hom_{{\Sp(2n)}}(\otimes^r V, \otimes^s V) %
= \{\phi: \otimes^r V\to \otimes^s V \mid g\cdot \phi(\mathbf v) %
= \phi(g\cdot \mathbf v), g\in\Sp(2n)\}.
\]
Note that the invariant subspace of $\otimes^r V$ is precisely
$\Hom_{\T_{\Sp(2n)}}(0,r) = \Hom_{{\Sp(2n)}}(\bC, \otimes^r V)$.

The connection between $\dc_{\Sp(2n)}$ and $\T_{\Sp(2n)}$ is
established by a functor we now define:
\begin{defn}\label{defn:ev}
  Let $\ev_{\Sp(2n)}$ be the evaluation functor
  $\dc_{\Sp(2n)}\rightarrow \T_{\Sp(2n)}$.  Explicitly,
  $\ev_{\Sp(2n)}$ sends the object $r\in\bN$ to $\otimes^rV$.  It is
  defined on the generators by
  \begin{align*}
    &\ev_{\Sp(2n)}\left(
      \begin{tikzpicture}[scale=0.6, baseline={(0,0.2)}, line width = 2pt]
        \fill[color=blue!20] (0,0) rectangle  (2,1);
        \draw (0.5,0) .. controls (0.5,0.5) and (1.5,0.5) .. (1.5,1);
        \draw (1.5,0) .. controls (1.5,0.5) and (0.5,0.5) .. (0.5,1);
      \end{tikzpicture}
    \right) = u\otimes v\mapsto -v\otimes u\\
    &\ev_{\Sp(2n)}\left(
      \begin{tikzpicture}[scale=0.6, baseline={(0,0.2)}, line width = 2pt]
        \fill[color=blue!20] (0,0) rectangle (2,1); %
        \draw (0.5,0) arc (180:0:0.5);
      \end{tikzpicture}
    \right) = 1 \mapsto \sum_i b_i\otimes \bar b_i\\
    &\ev_{\Sp(2n)}\left(
      \begin{tikzpicture}[scale=0.6, baseline={(0,0.2)}, line width = 2pt]
        \fill[color=blue!20] (0,0) rectangle (2,1); %
        \draw (0.5,1) arc (180:360:0.5);
      \end{tikzpicture}
    \right) = u\otimes v \mapsto \langle u, v\rangle.
  \end{align*}
and is then uniquely determined by the following three properties
\begin{itemize}
\item it is a functor, so $\ev_{\Sp(2n)}(x\cdot
  y)=\ev_{\Sp(2n)}(x)\circ\ev_{\Sp(2n)}(y)$,
\item it respects tensor products, so $\ev_{\Sp(2n)}(x\otimes
  y)=\ev_{\Sp(2n)}(x)\otimes\ev_{\Sp(2n)}(y)$,
\item it is linear.
\end{itemize}
\end{defn}

The first fundamental theorem for the symplectic group
can now be stated as follows:
\begin{thm}\cite{MR1503378},\cite[Theorem (6.1A)]{MR1488158}
  \label{thm:FFT-SFT-Brauer}
  For all $n>0$ the functor
  $\ev_{\Sp(2n)}\colon\dc_{\Sp(2n)}\rightarrow \T_{\Sp(2n)}$ is full.
\end{thm}

In the remainder of this section we provide an explicit description
of the kernel of $\ev_{\Sp(2n)}$ as an ideal in the Brauer category.
We will denote this ideal with $\Pf^{(n)}$ because of its intimate
connection to the Pfaffian.  Moreover we obtain a simple basis for
the vector space $\Hom_{\dc_{\Sp(2n)}/\Pf^{(n)}}(0, r)$.  This basis
is preserved by rotation, which we will use in
Section~\ref{sec:Brauer-CSP} to exhibit a cyclic sieving phenomenon.

The fundamental object involved is an idempotent $E(n+1)$ of the
Brauer algebra $D_{n+1}$ for which we have
$\ev_{\Sp(2n)}\big(E(n+1)\big) = 0$.  This element can be
characterised as follows.

Let $\rho$ be the one dimensional representation of $D_{n+1}$ which on diagrams
is given by
\begin{equation*}
 \rho(x) = \begin{cases}
 1 & \text{if $\pr(x)=n+1$} \\
 0 & \text{if $\pr(x)<n+1$},
\end{cases}
\end{equation*}
where $\pr(x)$ is the propagating number of a diagram $x$, i.e., the
number of strands connecting a point on the top edge with a point on
the bottom edge.

The element $E(n+1)$ of $D_{n+1}$ is determined, up to scalar multiple, by the properties
\begin{equation}\label{eq:ch}
 xE(n+1)=\rho(x)E(n+1)=E(n+1)x.
\end{equation}
It follows that $E(n+1)$ can be scaled so that it is idempotent and
these properties now determine $E(n+1)$.  It is clear that $E(n+1)$
is a central idempotent and that its rank equals one.

We will give two constructions of $E(n+1)$.  The first is as a simple
linear combination of diagrams.
\begin{defn}\label{defn:id}
 \begin{equation*}
 E(n+1) = \frac{1}{(n+1)!}\sum_{x\in D(n+1,n+1)}x
 \end{equation*}
\end{defn}

\begin{lemma} For $\delta=-2n$, the element $E(n+1)$ in
  Definition~\ref{defn:id} satisfies the properties \eqref{eq:ch}.
\end{lemma}

\begin{proof} If $\pr(x)=n+1$ then $x$ acts as a permutation on the
  set of diagrams $D(n+1,n+1)$, so this case is clear.

  Let $\cJ_r(p)$ be the ideal of the Brauer algebra $D_r$ generated
  by the diagrams with propagating number at most $p$.  The elements
  $u_i$, $1\le i\le n$ generate the ideal $\cJ_{n+1}(n)$ so it is
  sufficient to show that $u_i E(n+1)=0$ and $E(n+1)u_i=0$ for $1\le
  i\le n$. We will now show that $u_i E(n+1)=0$. The case
  $E(n+1)u_i=0$ is similar.

  It is clear that $u_i E(n+1)$ is a linear combination of diagrams
  which contain the pair $(i,i+1)$. Hence it is sufficient to show
  that the coefficient of each of these diagrams is 0. Let $x$ be a
  diagram which contains the pair $(i,i+1)$.  The set $\{y\in
  D(n+1,n+1): u_i y = x\}$ has cardinality $2n$ and there is
  precisely one diagram $z$ with $u_i z = \delta x$.  Therefore the
  coefficient of $x$ is $\delta+2n$, which vanishes for $\delta=-2n$.

The set of $2n$ diagrams is constructed as follows. The diagram $x$ contains
$n$ pairs other than $(i,i+1)$. For each of these pairs we construct two elements
of the set. Take the pair $(u,v)$, and replace the two pairs $(i,i+1)$ and $(u,v)$
by $(i,u)$ and $(i+1,v)$ and by $(i+1,u)$ and $(i,v)$ keeping all the remaining pairs.
\end{proof}

However it is not clear from this construction that
$\ev_{\Sp(2n)}\big(E(n+1)\big) = 0$.  We now give an alternative
construction which does make this clear.

\begin{defn}
  For $1\le i\le n$ define $u_i\in D(n+1,n+1)$ to consist of the
  pairs $(a,a')$ for $a\notin\{i,i+1\}$ together with the pairs
  $(i,i+1)$ and $(i',(i+1)')$ and define $s_i\in D(n+1,n+1)$ to consist
  of the pairs $(a,a')$ for $a\notin\{i,i+1\}$ together with the
  pairs $(i,(i+1)')$ and $(i',i+1)$:

  \begin{equation*}
    u_i = \raisebox{-1.25cm}
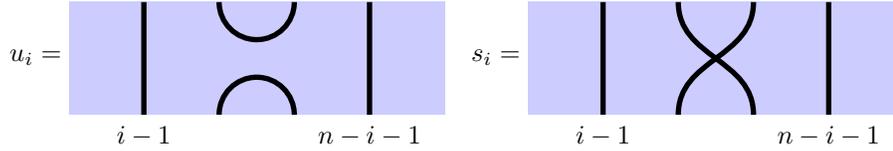
{%
    \begin{tikzpicture}[line width = 2pt]
      \fill[color=blue!20] (0,0) rectangle  (5,1.5);
      \draw (1,0) node[anchor=north] {$i-1$} -- (1,1.5) {};
      \draw (2,0) arc (180:0:0.5);
      \draw (2,1.5) arc (180:360:0.5);
      \draw (4,0) node[anchor=north] {$n-i-1$} -- (4,1.5) {};
    \end{tikzpicture}}
  \quad
    s_i = \raisebox{-1.25cm}{%
    \begin{tikzpicture}[line width = 2pt]
      \fill[color=blue!20] (0,0) rectangle  (5,1.5);
      \draw (1,0) node[anchor=north] {$i-1$} -- (1,1.5) {};
      \draw (2,0) .. controls (2,0.75) and (3,0.75) .. (3,1.5);
      \draw (3,0) .. controls (3,0.75) and (2,0.75) .. (2,1.5);
      \draw (4,0) node[anchor=north] {$n-i-1$} -- (4,1.5) {};
    \end{tikzpicture}}
  \end{equation*}
\end{defn}

\begin{defn} For $k\in \bZ$ and $1\le i\le n$ define $R_i(k)\in D_{n+1}$ by
\begin{equation*}
 R_i(k) = \frac1{k+1}\big( 1 + ks_i - \frac{2k}{\delta+2k-2}u_i \big)
\end{equation*}
\end{defn}

\begin{prop}\label{prop:yb} These elements satisfy
\begin{align*}
 R_i(h)R_{i+1}(h+k)R_i(k)&= R_{i+1}(k)R_i(h+k)R_{i+1}(h) \\
 R_i(h)R_j(k)&=R_j(k)R_i(h)\qquad\text{for $|i-j|>1$}
\end{align*}
\end{prop}

\begin{proof} The second relation is clear. The first relation is known as the Yang-Baxter
equation and is checked by a direct calculation. This calculation can be carried out using
generators and relations or by using a faithful representation.
\end{proof}

\begin{defn} The element $E(n+1)$ is defined recursively by
 \begin{equation*}
 E(n+1) = E(n)R_n(n)E(n)
 \end{equation*}
\end{defn}

There are other equivalent definitions, for example,
\begin{align*}
E(n+1) &= E(n)R_n(n)R_{n-1}(n-1)\dotsc R_1(1) \\
E(n+1) &= R_1(1)R_{2}(2)\dotsc R_n(n)E(n) \\
\end{align*}

The proof that these definitions are equivalent and the proof that
for $\delta=-2n$ this element satisfies the properties \eqref{eq:ch}
are both calculations using the Yang-Baxter equation
Proposition~\ref{prop:yb}.

\begin{prop}\label{prop:ker}
 \begin{equation*}
  \ev_{\Sp(2n)}(E(n+1))=0
 \end{equation*}
\end{prop}

\begin{proof}
  Because $\ev_{\Sp(2n)}$ is a functor, $\ev_{\Sp(2n)}(E(n+1))$ is an
  idempotent in $\T_{\Sp(2n)}$, too.  Moreover, $\ev_{\Sp(2n)}$
  preserves the trace of morphisms.  Recall
  that the (diagrammatic) trace $\tr_n$ of a diagram $\alpha$ in
  $D_n$ is defined as
  \[
  \eta_{2n} \cdot (\alpha\otimes\id_n) \cdot \eta^\ast_{2n}
  \]
  where $\eta_{2n}\in D(0,2n)$ is the diagram that consists of $n$
  nested arcs.  The following properties are easily verified.
  \begin{align*}
    \tr_{n+1}(\alpha\otimes\id_1) &= \delta \tr_n \alpha \\
    \tr_{n+1} \alpha s_n\beta &=\tr_n \alpha \beta\\
    \tr_{n+1} \alpha u_n \beta &= \tr_n \alpha \beta
  \end{align*}
  for $\alpha, \beta$ diagrams on $n$ strings.

  The rank of an idempotent is equal to its trace so it is sufficient
  to show that $\tr_{n+1} E(n+1)=0$ for $\delta=-2n$.

  We now compute the trace by expanding $R_{n}(n) = \frac{1}{n+1}(1+n
  s_n+n u_n)$:
  \[
  \tr_{n+1} E(n+1) %
  = \frac{1}{n+1}\left(\delta + n + n\right) %
  \tr_{n+1} E(n).
  \]
  Substituting $-2n$ for $\delta$ gives $\tr_{n+1} E(n+1)=0$.
\end{proof}

We will now describe the kernel of $\ev_{\Sp(2n)}$ in terms of
diagrammatic Pfaffians.
\begin{defn}[\protect{\cite[Definition 3.4 (b)]{MR1670662}}]\label{dfn:Pf-diagrammatic}
  Let $\Set S$ be a $2(n+1)$-element subset of $[r]\amalg [s]$ and
  let $f$ be a perfect matching of $[r]\amalg [s]\setminus\Set S$.
  Then the \Dfn{diagrammatic Pfaffian of order $2(n+1)$ corresponding
    to $f$} is
  \[
  \Pf(f) = \sum_{\text{$s$ a perfect matching of $\Set S$}} (s\cup f).
  \]
\end{defn}

Note that for $r,s=n+1$ and $f=\emptyset$ we have, by Definition~\ref{defn:id},
$E(n+1) = \frac1{(n+1)!} \Pf(\emptyset)$.

\begin{defn}\label{dfn:Pf-ideal}
  For $r,s\ge 0$, the subspace $\Pf^{(n)}(r,s)$ in the ideal
  $\Pf^{(n)}$ of $\dc_{\Sp(2n)}$ is spanned by the set
  \[
    \Pf^{(n)}(r,s) = \{ \Pf(f): \text{$f$ a perfect matching of a
      subset of $[r]\amalg [s]$}\\\text{ of cardinality $r+s-2(n+1)$
    } \}.
  \]
\end{defn}

\begin{defn}
  Let $\bev_{\Sp(2n)}\colon\dc_{\Sp(2n)}/\Pf^{(n)}\rightarrow
  \T_{\Sp(2n)}$ be the functor that factors $\ev_{\Sp(2n)}$ through
  the quotient.
\end{defn}

We can now state the main theorem of this section, also known as the
second fundamental theorem for the symplectic group:
\begin{thm}\label{thm:sft}
  The functor $\bev_{\Sp(2n)}\colon\dc_{\Sp(2n)}/\Pf^{(n)}\rightarrow
  \T_{\Sp(2n)}$ is an isomorphism of categories.
\end{thm}
\begin{proof}
  Since $\bev_{\Sp(2n)}$ is obviously bijective on objects and full by
  the first fundamental theorem, it is sufficient to show
  \[
  \dim\Hom_{\dc_{\Sp(2n)}/{\Pf^{(n)}}}(r,s)\leq \dim
  \Hom_{\T_{\Sp(2n)}}(r,s).
  \]
  This is achieved by combining Lemma~\ref{lem:Brauer-iso},
  Lemma~\ref{lem:Brauer-basis-oscillating-tableaux},
  Theorem~\ref{thm:Brauer-basis} and Lemma~\ref{lem:Brauer-Sundaram}
  below.
\end{proof}
We first restrict our attention to the invariant tensors:
\begin{lemma}\label{lem:Brauer-iso}
  We have isomorphisms of vector spaces
  \begin{align*}
    \Hom_{\dc_{\Sp(2n)}/{\Pf^{(n)}}}(r,s) &\cong \Hom_{\dc_{\Sp(2n)}/{\Pf^{(n)}}}(0,r+s)\\
    \Hom_{\T_{\Sp(2n)}}(r,s) &\cong \Hom_{\T_{\Sp(2n)}}(0,r+s).
  \end{align*}
\end{lemma}

Let us recall an indexing set for the basis of
$\Hom_{\T_{\Sp(2n)}}(0,r)$:
\begin{defn}
  An \Dfn{$n$-symplectic oscillating tableau} of length $r$ (and
  final shape $\emptyset$) is a sequence of partitions
  \[
  (\emptyset\!=\!\mu^0,\mu^1,\dots,\mu^r\!=\!\emptyset)
  \]
  such that the Ferrers diagrams of two consecutive partitions differ
  by exactly one cell and every partition $\mu^i$ has at most $n$
  non-zero parts.
\end{defn}

\begin{lemma}\label{lem:Brauer-basis-oscillating-tableaux}
  $\Hom_{\T_{\Sp(2n)}}(0,\otimes^r V)$ has a basis indexed by
  $n$-symplectic oscillating tableaux.
\end{lemma}
\begin{proof}
  This follows immediately from the branching rule for tensoring the
  defining representation with an irreducible representation of
  $\Sp(2n)$, see \cite[Theorem II]{MR0095209}.
\end{proof}

Next we exhibit a set of diagrams that span
$\Hom_{\dc_{\Sp(2n)}/\Pf^{(n)}}(0, r)$.  In fact, this is the key
observation.
\begin{defn}
  Let $d$ be a diagram in $D(0, r)$.  Then an \Dfn{$n$-crossing} in
  $d$ is a set of $n$ distinct strands such that every pair of
  strands crosses, i.e., $d$ contains strands
  $(a_1,b_1),\dotsc,(a_n,b_n)$ with
  $a_1<a_2<\dotsb<a_n<b_1<b_2<\dotsb<b_n$.  The diagram is
  $(n+1)$-noncrossing if it contains no $(n+1)$-crossing.
\end{defn}

\begin{thm}\label{thm:Brauer-basis}
  The $(n+1)$-noncrossing diagrams form a basis of
  $\Hom_{\dc_{\Sp(2n)}/\Pf^{(n)}}(0, r)$.
\end{thm}
\begin{proof}
  We only need to show that the set spans.  For each $f$ we write
  $\Pf(f)$ as a rewrite rule. The term that is singled out is the
  perfect matching of $\Set S$ in which every pair of strands
  crosses. The diagrams which cannot be simplified using these
  rewrite rules are the $(n+1)$-noncrossing diagrams.  The procedure
  terminates because the number of pairs of strands which cross
  decreases.
\end{proof}

We can now use a bijection due to Sundaram~\cite[Lemma 8.3]{MR2941115} to finish
the proof of our main theorem.
\begin{lemma}\label{lem:Brauer-Sundaram}
  For all $n$ and $r$ there is a bijection between the set of
  $n$-symplectic oscillating tableaux of length $r$ and the set of
  $(n+1)$-noncrossing diagrams in $D(0, r)$.
\end{lemma}

\section{Cyclic sieving phenomenon}
\label{sec:Brauer-CSP}

We now use the results obtained so far to exhibit instances of the
cyclic sieving phenomenon:
\begin{defn}
  Let $X$ be a finite set and let $\langle\rho\rangle$ be a cyclic
  group of order $r$ acting on $X$.  Let $P(q)$ be a polynomial with
  non-negative integer coefficients such that
  \[
  P(\zeta^d) = \left|\{x\in X\mid \rho^d x = x\}\right|,
  \]
  where $\zeta$ is a primitive $r$-th root of unity.  Then the triple
  \[
  \big(X, \rho, P(q)\big)
  \]
  exhibits the \Dfn{cyclic sieving phenomenon}.
\end{defn}

Recall that if $\rho\colon \fS_r\rightarrow \End(U)$ is a representation
of $\fS_r$ then the Frobenius character $\ch(U)$ is the homogeneous symmetric function
of degree $r$ given by
\begin{equation*}
 \ch(U) = \frac 1{r!} \sum_{\pi\in\fS_r} \tr \rho(\pi) p_{\lambda(\pi)}
\end{equation*}
where $p_\lambda$ is the power sum symmetric function and $\lambda(\pi)$
is the cycle type of $\pi$.

The fake degree, $\fd$, is a linear map from symmetric functions to
polynomials in $q$.  On the basis of Schur functions it is given by
\begin{equation*}
 \fd (s_\lambda) = \sum_T q^{\maj(T)}
\end{equation*}
where the sum is over standard tableaux of shape $T$ and $\maj(T)$ is the
major index of the tableau $T$.

A general technique to obtain a cyclic sieving polynomial $P(q)$ is
provided by the following result from~\cite{1512}:
\begin{thm}\label{cor:spr}
  Let $U$ be a representation of $\fS_r$ and let $X\subset U$ be a
  basis which is permuted by the long cycle $c$.  Then $\big(X,c,P(q)\big)$
exhibits the cyclic sieving phenomenon, where
\begin{equation*}
   P(q)= \fd\ch(U).
\end{equation*}
\end{thm}

The most straightforward application of this theorem is to
permutation representations of $\fS_r$. Here the representation and
the basis are given so it remains to determine the Frobenius
character and its fake degree polynomial.
In general, the action of the symmetric group on the space of
invariant tensors is not a permutation representation, that is, there
is no basis of the space of invariant tensors that is permuted by the
action of $\fS_r$.

As mentioned in the introduction, the Frobenius character of
$\Hom_{\Sp(V)}\big(W(\mu),\otimes^r V\big)$ was obtained in
\cite{MR2941115} (using combinatorics) and in \cite{MR933441} (using
representation theory).  For the special case of the invariant
tensors, a geometric proof can be found in Procesi~\cite[Equation
11.5.1.6]{MR2265844}.
\begin{lemma}\label{lem:it}
  \begin{equation*}
    \ch \Hom_{\T_{\Sp(2n)}}(0,2r) = \sum_{\substack{\lambda\vdash 2r\\
        \text{columns of even length}\\ \ell(\lambda)\le 2n}}
    s_{\lambda^t}.
  \end{equation*}
\end{lemma}
\begin{rem}
  The partitions indexing the Schur functions appearing in the
  Frobenius character are all transposed, since we defined $V$ to be
  an \emph{odd} vector space.
\end{rem}

\begin{thm}\label{thm:pmcsp} Let $X=X(r,n)$ be the set of $(n+1)$-noncrossing perfect
  matchings on $\{1,\dots,2r\}$ and let $\rho$ be the rotation map
  acting on $X$.  Let
  \begin{equation*}
    P(q) = \sum_{\substack{\lambda\vdash 2r\\
        \text{columns of even length}\\ \ell(\lambda)\le 2n}}
    \fd s_{\lambda^t}
  \end{equation*}
  Then the triple $\big(X, \rho, P(q)\big)$ exhibits the cyclic
  sieving phenomenon.
\end{thm}

The two extreme cases of Theorem~\ref{thm:pmcsp}, $n=1$ and $n\ge r$
are known.  Putting $n=1$, $X(r,1)$ is the set of non-crossing
perfect matchings.  The corresponding cyclic sieving phenomenon can
be found in \cite{MR2557880} and \cite{MR2519848}.  For $n\ge r$,
$X(r,n)$ is the set of all perfect matchings.  The Frobenius
character of this permutation representation was expanded by
Littlewood into Schur functions.  Using $\circ$ for plethystic
composition, we have
\begin{equation*}
h_r\circ h_2=\sum_{\substack{\lambda\vdash 2r\\ \text{rows of even length}}} s_{\lambda}.
\end{equation*}

\begin{rem} Let $T$ be an oscillating tableau. The descent set of $T$ is defined in \cite{MR3226822}
and is denoted by $\mathbf{Des}(T)$. The major index of $T$ is
\begin{equation*}
 \mathbf{maj}(T) = \sum_{i\in \mathbf{Des}(T)}i
\end{equation*}
Then the main result of \cite{MR3226822} is that
 \begin{equation*}
   P(q) = \sum_T q^{\mathbf{maj}(T)}
 \end{equation*}
where the sum is over $n$-symplectic oscillating tableaux of length $2r$ and weight $0$.
\end{rem}

Theorem~\ref{thm:pmcsp} can be generalised. Informally, we consider
the set of $k$-regular graphs (with loops prohibited but multiple
edges allowed) on $r$ vertices which are also $(n+1)$-noncrossing.
More precisely, consider $\{1,\dots kr\}$ as a cyclically ordered
set, partitioned into $r$ blocks of $k$ consecutive elements each.
Let $X=X(r,n,k)$ be the set of $(n+1)$-noncrossing perfect matchings
of $\{1,\dots kr\}$ such that there is no pair contained in a block
and if two pairs cross then the four elements are in four distinct
blocks, see Figure~\ref{fig:symmetric-powers} for an example.
Finally, let $\rho$ be rotation by $k$ points, i.e.,
$\rho(i)=i+k-1\pmod{kr}+1$.

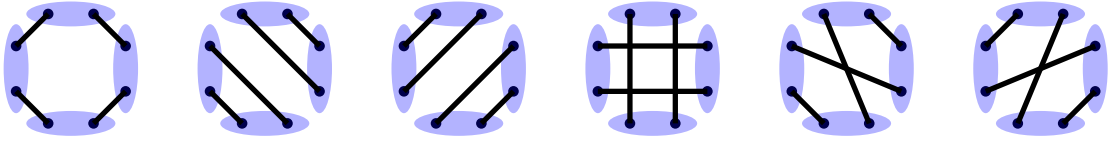
\begin{figure}
  \centering
  \begin{tikzpicture}[line width=2pt]
    \node[draw=none,Xsize,regular polygon,regular polygon sides=8] (a) {};
    \foreach \x in {1,2,...,8} 
    \fill (a.corner \x) circle[radius=2pt];
    \node[fit=(a.corner 3)(a.corner 4), block] {};
    \node[fit=(a.corner 5)(a.corner 6), block] {};
    \node[fit=(a.corner 7)(a.corner 8), block] {};
    \node[fit=(a.corner 1)(a.corner 2), block] {};
    \draw (a.corner 1) to (a.corner 8);
    \draw (a.corner 2) to (a.corner 3);
    \draw (a.corner 4) to (a.corner 5);
    \draw (a.corner 6) to (a.corner 7);
  \end{tikzpicture}
  \qquad
  \begin{tikzpicture}[line width=2pt]
    \node[draw=none,Xsize,regular polygon,regular polygon sides=8] (a) {};
    \foreach \x in {1,2,...,8}
    \fill (a.corner \x) circle[radius=2pt];
    \node[fit=(a.corner 3)(a.corner 4), block] {};
    \node[fit=(a.corner 5)(a.corner 6), block] {};
    \node[fit=(a.corner 7)(a.corner 8), block] {};
    \node[fit=(a.corner 1)(a.corner 2), block] {};
    \draw (a.corner 1) to (a.corner 8);
    \draw (a.corner 2) to (a.corner 7);
    \draw (a.corner 3) to (a.corner 6);
    \draw (a.corner 4) to (a.corner 5);
  \end{tikzpicture}  
  \qquad
  \begin{tikzpicture}[line width=2pt]
    \node[draw=none,Xsize,regular polygon,regular polygon sides=8] (a) {};
    \foreach \x in {1,2,...,8}
    \fill (a.corner \x) circle[radius=2pt];
    \node[fit=(a.corner 3)(a.corner 4), block] {};
    \node[fit=(a.corner 5)(a.corner 6), block] {};
    \node[fit=(a.corner 7)(a.corner 8), block] {};
    \node[fit=(a.corner 1)(a.corner 2), block] {};
    \draw (a.corner 1) to (a.corner 4);
    \draw (a.corner 2) to (a.corner 3);
    \draw (a.corner 5) to (a.corner 8);
    \draw (a.corner 6) to (a.corner 7);
  \end{tikzpicture}  
  \qquad
  \begin{tikzpicture}[line width=2pt]
    \node[draw=none,Xsize,regular polygon,regular polygon sides=8] (a) {};
    \foreach \x in {1,2,...,8}
    \fill (a.corner \x) circle[radius=2pt];
    \node[fit=(a.corner 3)(a.corner 4), block] {};
    \node[fit=(a.corner 5)(a.corner 6), block] {};
    \node[fit=(a.corner 7)(a.corner 8), block] {};
    \node[fit=(a.corner 1)(a.corner 2), block] {};
    \draw (a.corner 1) to (a.corner 6);
    \draw (a.corner 2) to (a.corner 5);
    \draw (a.corner 3) to (a.corner 8);
    \draw (a.corner 4) to (a.corner 7);
  \end{tikzpicture}  
  \qquad
  \begin{tikzpicture}[line width=2pt]
    \node[draw=none,Xsize,regular polygon,regular polygon sides=8] (a) {};
    \foreach \x in {1,2,...,8}
    \fill (a.corner \x) circle[radius=2pt];
    \node[fit=(a.corner 3)(a.corner 4), block] {};
    \node[fit=(a.corner 5)(a.corner 6), block] {};
    \node[fit=(a.corner 7)(a.corner 8), block] {};
    \node[fit=(a.corner 1)(a.corner 2), block] {};
    \draw (a.corner 1) to (a.corner 8);
    \draw (a.corner 2) to (a.corner 6);
    \draw (a.corner 3) to (a.corner 7);
    \draw (a.corner 4) to (a.corner 5);
  \end{tikzpicture}  
  \qquad
  \begin{tikzpicture}[line width=2pt]
    \node[draw=none,Xsize,regular polygon,regular polygon sides=8] (a) {};
    \foreach \x in {1,2,...,8}
    \fill (a.corner \x) circle[radius=2pt];
    \node[fit=(a.corner 3)(a.corner 4), block] {};
    \node[fit=(a.corner 5)(a.corner 6), block] {};
    \node[fit=(a.corner 7)(a.corner 8), block] {};
    \node[fit=(a.corner 1)(a.corner 2), block] {};
    \draw (a.corner 1) to (a.corner 5);
    \draw (a.corner 2) to (a.corner 3);
    \draw (a.corner 4) to (a.corner 8);
    \draw (a.corner 6) to (a.corner 7);
  \end{tikzpicture}  
  \caption{The six elements of $X(4,2,2)$.}
  \label{fig:symmetric-powers}
\end{figure}

For $n>kr$, the set $X(r,n,k)$ is independent of $n$ and can be
identified with the set of $k$-regular graphs with multiple edges
allowed but loops prohibited. These $k$-regular graphs would normally
be considered as a permutation representation of $\fS_r$ acting
by permuting the vertices. Here the action of $\fS_r$ is a linear
action on the space with basis $X(r,n,k)$.  Example \ref{ex:diff}
below shows that these actions are different.

\begin{lemma}
  The invariant tensors corresponding to the diagrams $X(r,n,k)$ form
  a basis of the space of $\Sp(2n)$-invariants in the $r$-th tensor
  power of the $k$-th symmetric power of the defining representation
  of $\Sp(2n)$.
\end{lemma}

The Frobenius character of this representation is given by an
application of \cite[Theorem 1]{6111}.
\begin{lemma}
  The Frobenius character of the space of $\Sp(2n)$-invariants in the
  $r$-th tensor power of the $k$-th symmetric power of the defining
  representation of $\Sp(2n)$ is
  \begin{equation}\label{eqn:invs}
    \sum_{\substack{\lambda\vdash kr\\
        \text{columns of even length}\\ \ell(\lambda)\le 2n}}
    \big\langle  h_r\big(X\cdot e_k(Y)\big), %
    s_{\lambda^t}(Y)  \big\rangle_Y
  \end{equation}
  where $e_k$ is the $k$-th elementary symmetric function and
  $\langle\ ,\ \rangle_Y$ denotes the scalar product of symmetric
  functions with respect to the alphabet $Y$.
\end{lemma}

\begin{thm}
  Let $X=X(r,n,k)$ and let $\rho$ be the rotation map.  Define $P(q)$
  to be the fake degree of the symmetric function \eqref{eqn:invs}.
  Then $(X,\rho,P)$ exhibits the cyclic sieving phenomenon.
\end{thm}

The case $k=1$ is Theorem~\ref{thm:pmcsp} and the case $k=2$ is
related to the invariant tensors of the adjoint representation in
\cite{MR792707}. The case $n=1$ is implicit in \cite{MR1446615} and
the case $n=2$ is related to the $C_2$ webs in \cite{MR1403861}.

Instead of the $k$-th symmetric power of $V$ we can also consider the
$k$-th fundamental representation of $\Sp(2n)$.  Again using
\cite[Theorem 1]{6111} we obtain the following expression for the
Frobenius character of its tensor powers.
\begin{lemma}
  The Frobenius character of the space of $\Sp(2n)$-invariants in the
  $r$-th tensor power of the $k$-th fundamental representation of
  $\Sp(2n)$ is
  \begin{equation}\label{eqn:invf}
    \sum_{\substack{|\lambda|\le kr\\
        \text{columns of even length}\\ \ell(\lambda)\le 2n}}
    \big\langle  h_r\big(X\cdot (h_k-h_{k-2})(Y)\big), %
    s_{\lambda^t}(Y)  \big\rangle_Y
  \end{equation}
\end{lemma}
Note that in general, this space cannot have a basis invariant under
cyclic rotation.  For example, for $n=1$, $k=3$ and $r=2$ its
Frobenius character evaluates to $s_{1,1}$ with fake degree
polynomial equal to $q$, which is not a cyclic sieving polynomial.

However, let us compare the two Frobenius characters for $n>kr$.
Setting $H=1+h_1+h_2+\cdots$ we obtain for the symmetric powers
\begin{equation}\label{eqn:regs}
  \big\langle  h_r\big(X\cdot e_k(Y)\big), %
  H(h_2(Y))  \big\rangle_Y%
\end{equation}
whereas for the fundamental representation the expression becomes
\begin{equation}\label{eqn:regf}
  \big\langle  h_r\big(X\cdot (h_k-k_{k-2})(Y)\big), %
  H(h_2(Y))  \big\rangle_Y
\end{equation}

For $n>kr$, the set $X(r,n,k)$ is independent of $n$ and can be
identified with the set of $k$-regular graphs with multiple edges
allowed but loops prohibited.  This is a species whose Frobenius
character is given by equation~\eqref{eqn:regf}, see Travis~\cite{MR2698697}.

The action of the symmetric group in the first case is defined using
a sign and is therefore not a permutation representation, as
illustrated by the following example.
\begin{ex}\label{ex:diff} Putting $k=2$ and $r=6$, the formula \eqref{eqn:regs} for
  invariant tensors gives
  \begin{equation}
    \frac{1}{72}(13p_{1^6} + 12p_{21^4} + 63p_{2^21^2} +
    54p_{2^3} + 4p_{31^3} - 12p_{321} + 28p_{3^2} + 18p_{41^2}
    + 36p_{42} + 36p_{6})
  \end{equation}
  and the formula \eqref{eqn:regf} for regular graphs gives
  \begin{equation*}
    \frac{1}{72}(13p_{1^6} + 24p_{21^4} + 63p_{2^21^2} +
    54p_{2^3} + 4p_{31^3} + 12p_{321} + 28p_{3^2} + 18p_{41^2}
    + 36p_{42} + 36p_{6})
 \end{equation*}
\end{ex}

\section{Partitions and directed matchings}
The discussion in the previous sections has been about the Brauer category
and the categories of invariant tensors for the vector representation of $\Sp(2n)$.
In this section we discuss two other examples where we have a diagram category
which is related to a sequence of categories of invariant tensors.

In the first example we consider the linear representation $V$ of the
symmetric group $\fS_n$ associated to its defining representation.
The diagram category in this case is the partition category studied
in~\cite{MR2143201}.

Let $I(r,n)=\Hom_{\fS_n}(\bC,\otimes^{r}V)$ be the space of invariant
tensors with the action of $\fS_{r}$ and let $X(r,n)$ be the set of
set partitions of $\{1,\dots,r\}$ into at most $n$ blocks.  Then
there is a map $X(r,n)\rightarrow I(r,n)$ such that the image is a
basis.  Furthermore this basis is preserved by the action of $\fS_r$.
It follows that the Frobenius character is the homogeneous component
of degree $r$ in $h_n\circ H$, where $H=1+h_1+h_2+\cdots$.

As in the case of the Brauer category, this result can be
generalised, by considering the tensor powers of the $k$-th symmetric
power of the defining representation.  The resulting set $X(r,n,k)$
is then the set of multiset partitions into at most $n$ blocks of the
multiset $\{1,\dots,1,2,\dots,2,\dots, r,\dots,r\}$, each label
occuring precisely $k$ times.  The symmetric group $\fS_r$ acts on
this set by permuting the labels.  The Frobenius character can be
obtained using~\cite[Theorem 1]{6111} or alternatively with the
calculus of species and is given by
\[
\big\langle h_r\big(X\cdot h_k(Y)\big), %
h_n(H(Y)) \big\rangle_Y.
\]

In the second example we consider the adjoint representation $V$ of
the general linear group $\GL(n)$.  In this case the diagrams are
Brauer diagrams in which every edge is directed.

For $n\geq 2r$ let $X(r)$ be the set of permutations with the
conjugation action of $\fS_r$.  Its Frobenius character is obtained
in \cite{MR768993} using character calculations. Our methods give a
simple proof. The Frobenius character is
\begin{equation*}
 \sum_{\lambda\vdash r} p_\lambda =  \sum_{\lambda\vdash r} s_\lambda \ast s_\lambda
\end{equation*}
where $\ast$ denotes the Kronecker or inner product.

Let $I(r,n)=\Hom_{\GL(n)}(\bC,\otimes^{r}V)$ be the space of invariant
tensors with the action of $\fS_{r}$.  Considering the diagrams
we obtain the new result that the Frobenius character of this representation is
\begin{equation*}
  \sum_{\substack{\lambda\vdash r \\ \ell(\lambda)\le n}} s_\lambda \ast s_\lambda
\end{equation*}

However, for $n<2r$ it is an open problem to determine a set $X(r,n)$
with a map $X(r,n)\rightarrow I(r,n)$ such that the image is a basis
and is preserved by the action of the long cycle.  So we are not able
to exhibit an example of the cyclic sieving phenomenon in this case.

\printbibliography
\end{document}